\documentclass{amsart}
\newtheorem{theorem}{Theorem}[section]
\newtheorem{lemma}[theorem]{Lemma}

\theoremstyle{definition}

\theoremstyle{remark}
\newtheorem{remark}[theorem]{Remark}

\numberwithin{equation}{section}

\begin{document}

\title[Affine-geometric Wirtinger inequality]{Affine-geometric Wirtinger inequality}

\author[M. N. Ivaki]{Mohammad N. Ivaki}
\address{Department of Mathematics and Statistics,
  Concordia University, Montreal, QC, Canada, H3G 1M8}
\curraddr{}
\email{mivaki@mathstat.concordia.ca}
\thanks{}


\date{}

\dedicatory{}

\commby{Lei Ni}

\begin{abstract}
A generalization of the affine-geometric Wirtinger inequality for curves to hypersurfaces is given.
\end{abstract}

\maketitle
\section{Introduction}
\label{intro}
Suppose that $\mathcal{M}$  belongs to the set of smooth, strictly convex bodies in $\mathbb{R}^{n}$. Let $X:\partial{\mathcal{M}}\to\mathbb{R}^{n}$ be a smooth embedding of $(\partial{\mathcal{M}},g),$ the boundary of $\mathcal{M}$, with the induced Riemannian metric $g$. The affine normal vector is a vector field $\xi$ transverse to $\partial{\mathcal{M}}$ and invariant under volume-preserving affine transformations of $\mathbb{R}^n$.
The support function and the second fundamental form of $\partial{\mathcal{M}}$ are defined, respectively, by
$$s(z):=\langle z,X((-\nu)^{-1}(z))\rangle,~~~~h_{ij}=\hat{\nabla}^2_{ij}s+\hat{g}_{ij}s,$$
for each $z\in\mathbb{S}^{n-1},$
where $\hat{g}_{ij}$ is the standard metric on $\mathbb{S}^{n-1}$  and $\hat{\nabla}$ is the Levi-Civita connection induced by $\hat{g}$.
Here $\nu$ is the inward unit normal, thus $-\nu:\partial\mathcal{M}\to\mathbb{S}^{n-1}$ defines the Gauss map.
 We denote the Gauss curvature of $\partial{\mathcal{M}}$ by $\mathcal{K}$ and remark that, as a function on the unit sphere,  is related to the support function of the convex body by $$\frac{1}{\mathcal{K}}=\det_{\hat{g}}(\hat{\nabla}^2_{ij}s+\hat{g}_{ij}s):=\frac{\det h_{ij}}{\det{\hat{g}_{ij}}}.$$
 We need to recall also a few definitions from affine differential geometry, for an excellent and detailed exposition we point to the reference \cite{LT}.
  The affine metric of $\partial\mathcal{M}$ is defined by $\bar{g}_{ij}:=\frac{h_{ij}}{\mathcal{K}^{1/(n+1)}}$. We denote its induced Levi-Civita connection by $\bar{\nabla}$ and the induced volume measure by $\displaystyle d\bar{\mu}:=\sqrt{\bar{g}_{ij}}~dx^1\cdot\cdot \cdot  dx^{n-1}$, in a local coordinate chart $(x_1,...,x_{n-1})$. We also denote the affine curvature of $\partial{\mathcal{M}}$ by $A_{ij}$, where $A_{ij}$ is the second fundamental form of the affine metric with respect to $\bar{\nabla}$. The trace of $A_{ij}$ with respect to $\bar{g}_{ij}$, denoted by $H$, is called the affine mean curvature. In \cite{BA3}, Andrews proved the following inequality, called the affine-geometric Wirtinger inequality, for smooth, strictly convex curves $\gamma$ embedded in $\mathbb{R}^2$:
  \begin{equation*}
\int\limits_{\gamma}F^2 \eta ds\leq \frac{1}{2}\frac{\left(\int\limits_{\gamma}Fds\right)^2}{{\mbox{Area}}(\gamma)}
+\int\limits_{\gamma}F_{s}^2ds,
\end{equation*}
where $F:\gamma\rightarrow\mathbb{R}$ is any smooth function along the curve, $s$ is the affine arc-length of $\gamma$, $\eta$ is its affine curvature, and ${\mbox{Area}}(\gamma)$ is the area of the planar region enclosed by $\gamma.$\\
This inequality can be generalized to smooth, strictly convex hypersurfaces in $\mathbb{R}^n$:
\begin{equation*}
\int\limits_{\partial\mathcal{M}}F^2Hd\bar{\mu}\leq \frac{n-1}{n}\frac{\left(\int\limits_{\partial\mathcal{M}}Fd\bar{\mu}\right)^2}{\mbox{Vol}(\mathcal{M})}
+\int\limits_{\partial\mathcal{M}}\left|\bar{\nabla}F\right|^2_{\bar{g}}d\bar{\mu},
\end{equation*}
for any smooth function $F:\partial{M}\to\mathbb{R}$. Characterization of the equality case is also given.
Affine-geometric Wirtinger inequality is essentially a translation of Minkowski's mixed volume to the language of affine differential geometry. This will be made explicit in the next section.

For simplicity, we added the relevance of this inequality in affine geometric flows and in the $L_{-2}$ Minkowski's problem in the concluding remark at the end of the paper.

\section{Main results}
\begin{theorem}\label{thm: generalized} Suppose that $f$ is a smooth function defined on $(\partial\mathcal{M},g)$. Then the following inequality holds
\begin{equation}\label{e: generalized}
h^{ij}\left(\hat{\nabla}_{ij}^2f+\hat{g}_{ij}f\right)=
\bar{\Delta}\frac{f}{\mathcal{K}^{1/(n+1)}}+\frac{f}{\mathcal{K}^{1/(n+1)}}H,
\end{equation}
where $h^{ij}$ is the inverse of the second fundamental form $h_{ij}.$
\end{theorem}
\begin{proof} This is a convenient identity which can be derived in several ways. We present here a proof of this identity based on a geometric flow. This approach gives an interesting geometric interpretation of this identity: The identity is the result of computing the instantaneous rate of change of the Gauss curvature under an arbitrary variation, where in the affine setting the computation depends on the component of the variation in the direction
of the affine normal vector, and in the Gauss map parametrization it depends on the component of the variation in the direction of the
Euclidean normal vector.

Without loss of generality, we can assume that $f$ is a positive function.  We define $F:=\frac{f}{\mathcal{K}^{1/(n+1)}}.$ Suppose that we evolve $\partial\mathcal{M}$ by the following weighted affine normal flow
$$\partial_tX(z,t)=F\xi(z,t).$$

 Short time existence of the flow follows from the standard theory of parabolic equations and this is solely what we will need.
 \begin{lemma}\label{lem: ev}
 Define $\nu$ to be the inward unit normal to $(\partial\mathcal{M},g)$. The following evolution equations hold for as long as the  weighted affine normal flow exists.
 \begin{enumerate}
   \item In the Euclidean parametrization ($\mathcal{E}$), we have $$\displaystyle\partial_t\nu=-\mathcal{K}^{1/(n+1)}g^{ij}\partial_i FX_{.,j},$$
   \item and $$\displaystyle\partial_t\mathcal{K}=\mathcal{K}\bar{\Delta}F+\mathcal{K}HF+\langle \bar{\nabla}F,\bar{\nabla}\mathcal{K}\rangle_{\bar{g}}.$$
   \item In the Gauss parametrization ($\mathcal{G}$), we have
    $$\displaystyle\partial_ts=-f,$$
    \item and $$\displaystyle\partial_t\mathcal{K}=\mathcal{K}h^{ij}\left(\hat{\nabla}_{ij}^2f+\hat{g}_{ij}f\right).$$
\end{enumerate}
\end{lemma}
The proof of the this lemma is entirely calculatory and we include it in the appendix.
\begin{lemma} Let $\varepsilon>0$ be a small number. Suppose that $Q:\partial\mathcal{M}^n\times [0, \varepsilon)\to \mathbb{R}$ and
$\bar{Q}:\mathbb{S}^{n-1}\times [0, \varepsilon)\to\mathbb{R}$
are smooth functions related by the following equation
$$\bar{Q}(-\nu(x,t),t)=Q(x,t).$$
Then the evolution equations of $Q$ and $\bar{Q}$ along the flow are related by
$$\partial_tQ=\partial_t\bar{Q}+\bar{g}^{ij}\partial_j\bar{Q}\partial_iF .$$
\end{lemma}
\begin{proof} Recall that $\partial_m\nu=-h_{mp}g^{pl}X_{.,l}$ and $\hat{g}^{im}g^{jn}h_{mn}=h^{ij}.$
\begin{align*}
\partial_tQ&=\partial_t\bar{Q}-\langle \hat{\nabla}\bar{Q}, \partial_t\nu\rangle\\
&=\partial_t\bar{Q}-\langle -\partial_k\bar{Q}\hat{g}^{km}\partial_m\nu, \partial_t\nu\rangle\\
&=\partial_t\bar{Q}+\langle \partial_k\bar{Q}\hat{g}^{km}\partial_m\nu, -\mathcal{K}^{1/(n+1)}g^{ij}\partial_i FX_{.,j}\rangle\\
&=\partial_t\bar{Q}- \partial_k\bar{Q}\partial_i F\hat{g}^{km}g^{ij}\mathcal{K}^{1/(n+1)}\langle\partial_m\nu, X_{.,j}\rangle\\
&=\partial_t\bar{Q}+ \partial_k\bar{Q}\partial_i F\hat{g}^{km}g^{ij}\mathcal{K}^{1/(n+1)}
\langle h_{mp}g^{pl}X_{.,l}, X_{.,j}\rangle\\
&=\partial_t\bar{Q}+ \partial_k\bar{Q}\partial_i F\mathcal{K}^{1/(n+1)}\hat{g}^{km}g^{ij}
h_{mp}g^{pl}g_{lj}\\
&=\partial_t\bar{Q}+ \partial_k\bar{Q}\partial_i F\mathcal{K}^{1/(n+1)}\hat{g}^{km}g^{ij}h_{mj}\\
&=\partial_t\bar{Q}+ \partial_k\bar{Q}\partial_i F\mathcal{K}^{1/(n+1)}h^{ki}\\
&=\partial_t\bar{Q}+ \bar{g}^{ij}\partial_j\bar{Q}\partial_i F.
\end{align*}
\end{proof}

To finish the proof  of Theorem \ref{thm: generalized}, we note that by the previous lemmas we have the following relation between the evolution equation of $\mathcal{K}$ in the Euclidean parametrization and in the Gauss parametrization:
$$\left(\partial_t\mathcal{K}\right)_{\mathcal{E}}
=\left(\partial_t\mathcal{K}\right)_{\mathcal{G}}+\langle \bar{\nabla}F,\bar{\nabla}\mathcal{K}\rangle_{\bar{g}}.$$
\end{proof}

Define $A[f]:=\hat{\nabla}_{ij}^2f+\hat{g}_{ij}f$, for any smooth function $f$ on $\mathbb{S}^{n-1}.$ The mixed curvature function of $n-1$ convex
bodies $K_1, K_2,\cdots, K_{n-1}$ with smooth support functions $s_1, s_2,\cdots, s_{n-1}$, denoted by
$Q[s_1, s_2,\cdots, s_{n-1}]$, is defined as follows
$$Q[s_1, s_2,\cdots, s_{n-1}]:=\frac{1}{(n-1)!}
\sum_{\sigma,\tau\in P_{n-1}}(-1)^{\mbox{sgn}(\sigma)+\mbox{sgn}(\tau)}A[s_1]_{\tau(1)}^{\sigma(1)}\cdots A[s_{n-1}]_{\tau(n-1)}^{\sigma(n-1)}.$$
Here $P_{n-1}$ is the group of permutations on $n-1$ objects.
Let $K$ be a convex body with smooth support function $s$. The mixed volume of $K, K_1,\cdots, K_{n-1}$, denoted by $V[K, K_1,\cdots, K_{n-1}]$ is defined as follows
$$V[K, K_1,\cdots, K_{n-1}]=\int\limits_{\mathbb{S}^{n-1}}sQ[s_1, s_2,\cdots, s_{n-1}]d\mu_{\mathbb{S}^{n-1}}.$$
Alternatively, we will write $V[s, s_1,\cdots, s_{n-1}]$ instead of $V[K, K_1,\cdots, K_{n-1}].$
By Minkowski's mixed volume inequality, we have
$$V[h,h,s_1,\cdots,s_{n-2}]\leq \frac{~V[s,h,s_1,\cdots,s_{n-2}]^2}{V[s,s,s_{1},\cdots,s_{n-2}]}$$
for any smooth function $h$ on $\mathbb{S}^{n-1}.$ Equality holds if and only if $h=cs+d$ for some constants $c, d\in\mathbb{R}.$ For more on the mixed curvature function, mixed volumes of convex bodies and the Minkowski inequality, we refer the reader to  \cite{schneider}.

\begin{remark}
Fix $s_1,\cdots,s_{n-2}$, then $Q[f]:= Q[f,s_1,\cdots, s_{n-2}]$ is a nondegenerate second-order linear elliptic operator on smooth functions on $\mathbb{S}^{n-1}$, expressed in local coordinates in the following form:
$$Q[f]=\sum_{i,j}\dot{Q}^{ij}\left(\hat{\nabla}_{ij}^2f+\hat{g}_{ij}f\right),$$
where $\dot{Q}$ is a positive definite matrix at each point of $\mathbb{S}^{n-1}$ which depends only
on the functions $s_1,\cdots,s_{n-2}$.
\end{remark}
\begin{lemma}\label{re} Let $\mathcal{M}$ be a smooth convex body with support function $s$. Then, for any smooth function $f:\partial\mathcal{M}\to\mathbb{R}$, we have
$$Q[f, \underbrace{s,\cdots,s}_{n-2~ \mbox{times}}]=\frac{1}{n-1}\frac{h^{ij}}{\mathcal{K}}\left(\hat{\nabla}_{ij}^2f+\hat{g}_{ij}f\right).$$
\end{lemma}
\begin{proof} By the definition of mixed curvature, we have
$$Q[f, s,\cdots, s]:=\frac{1}{(n-1)!}
\sum_{\sigma,\tau\in P_{n-1}}(-1)^{\mbox{sgn}(\sigma)+\mbox{sgn}(\tau)}A[f]_{\tau(1)}^{\sigma(1)}A[s]_{\tau(2)}^{\sigma(2)}\cdots A[s]_{\tau(n-1)}^{\sigma(n-1)}.$$
Now the claim can be easily verified by choosing appropriate coordinate patches in which $h_{ij}$ is a diagonal matrix.
\end{proof}

We are now ready to prove the affine-geometric Wirtinger inequality on smooth strictly convex hypersurfaces in $\mathbb{R}^n$.
\begin{theorem}(The affine-geometric Wirtinger inequality) Suppose that $F:\partial\mathcal{M}\rightarrow\mathbb{R}$ is a smooth function on $\mathcal{M}$, an arbitrary  smooth strictly convex hypersurface in $\mathbb{R}^n$. Then
\begin{equation}\label{ie: affine-Wirtinger inequality}
\int\limits_{\partial\mathcal{M}}F^2Hd\bar{\mu}\leq \frac{n-1}{n}\frac{\left(\int_{\partial\mathcal{M}}Fd\bar{\mu}\right)^2}{\mbox{Vol}(\mathcal{M})}
+\int\limits_{\partial\mathcal{M}}\left|\bar{\nabla}F\right|^2_{\bar{g}}d\bar{\mu}.
\end{equation}
Furthermore, with the earlier notations, equality holds if and only if  $\displaystyle F=\frac{1}{\mathcal{K}^{1/(n+1)}}(cs+d)$ for some constants $c, d\in\mathbb{R}.$
\end{theorem}
\begin{proof} Define $f:=F\mathcal{K}^{1/(n+1)}$, hence using Lemma \ref{re} and Theorem \ref{thm: generalized}, we have
\begin{align*}
V\left[f,f,s,\cdots,s\right]&=\int\limits_{\mathbb{S}^{n-1}}fQ[f, s,\cdots, s]d\mu_{\mathbb{S}^{n-1}}\\
&=\int\limits_{\partial\mathcal{M}}\frac{f}{\mathcal{K}^{1/(n+1)}}Q[f, s,\cdots, s]\mathcal{K}d\bar{\mu}\\
&=\frac{1}{n-1}\int\limits_{\partial\mathcal{M}}\frac{f}{\mathcal{K}^{1/(n+1)}}\left(\bar{\Delta}\frac{f}{\mathcal{K}^{1/(n+1)}}
+\frac{f}{\mathcal{K}^{1/(n+1)}}H\right)d\bar{\mu}\\
&=\frac{1}{n-1}\int\limits_{\partial\mathcal{M}}\left(-\left|\bar{\nabla}F\right|_{\bar{g}}^2+F^2H\right)
d\bar{\mu}.
\end{align*}
 We used integration by parts on the last line. Finally, by Minkowski's mixed volume inequality, and the identities $$V[f,s,\cdots,s]=\int\limits_{\partial\mathcal{M}}Fd\bar{\mu}\ \ \ {\mbox{and}}\ \ \
 V[s,s,\cdots,s]=n\mbox{Vol}(\mathcal{M}),$$
 we obtain the desired inequality
(\ref{ie: affine-Wirtinger inequality}). The equality condition of our theorem comes from the equality condition of Minkowski's mixed volume inequality.
\end{proof}

\begin{remark} We would like to state few remarks on the importance of the affine-geometric Wirtinger inequality.
\begin{enumerate}
  \item Using affine Wirtinger inequality one can prove that $p$-affine isoperimetric ratio, $\displaystyle\frac{\int\limits_{\mathbb{S}^{n-1}}\frac{s}{\mathcal{K}}\left(\frac{\mathcal{K}}{s^{n+1}}\right)^{\frac{p}{n+p}}
  d\mu_{\mathbb{S}^{n-1}}}{Vol^{\frac{n-p}{n+p}}}$, is increasing along the $p$-centro affine normal flow
$$\partial_ts=-s\left(\frac{\mathcal{K}}{s^{n+1}}\right)^{\frac{p}{p+n}},~ p\geq 1$$
starting the flow from any smooth, origin symmetric convex hypersurface in $\mathbb{R}^{n}.$ This flow is first introduced by Stancu \cite{S} for the purpose of finding new affine invariant quantities in centro-affine convex geometry.
The long time behavior of the flow  in $\mathbb{R}^{2}$ was studied in \cite{Ivaki}. Using monotonicity of $p$-affine isoperimetric ratio, it was shown that this flow evolves any smooth, origin symmetric convex curve to an ellipsoid in Banach Mazur distance.
The long time behavior of $p$-centro affine normal flows remains open in higher dimensions.
  \item By means of affine Wirtinger inequality one can prove that the weighted $p$-affine isoperimetric ratio, $\displaystyle\frac{\int\limits_{\mathbb{S}^{n-1}}\Phi\frac{s}{\mathcal{K}}\left(\frac{\mathcal{K}}{s^{n+1}}\right)^{\frac{p}{n+p}}
  d\mu_{\mathbb{S}^{n-1}}}{Vol^{\frac{n-p}{n+p}}}$, is increasing along the weighted $p$-centro affine normal flow
$$\partial_ts=-\Phi s\left(\frac{\mathcal{K}}{s^{n+1}}\right)^{\frac{p}{p+n}},~ p\geq 1$$
starting the flow from any smooth, origin symmetric convex hypersurface in $\mathbb{R}^{n}.$ Here $\Phi$ is an even, positive smooth function on $\mathbb{S}^{n-1}.$
Note that in general the solution to the normalized weighted $p$-centro affine normal flow is not convergent as $\Phi$ must satisfy some obstructions.
For example in $\mathbb{R}^2$ it was proved in \cite{Ivaki2} that $\Psi$ must have at least eight critical points. Using this flow it was shown  therein that for any given even smooth positive function on $\mathbb{S}^1$ there is a smooth, origin symmetric convex curve with support function $s$ that \emph{almost} solves
$$s^3(s_{\theta\theta}+s)=\Psi.$$
More precisely, given an $\varepsilon>0$ there is a smooth, origin symmetric convex curve $\gamma$ such that
$$||s^3(s_{\theta\theta}+s)-\Psi||<\varepsilon.$$
Let $\hat{\nabla}$ denotes the standard Levi-Civita connection of $\mathbb{S}^{n-1}$.  Finding necessary and sufficient condition on $\Psi\in C^{\infty}(\mathbb{S}^{n-1})$ to guarantee the existence of a convex body with support function $s$ that solves
$$s^{n+1}\det\left(\hat\nabla^2_{ij}+\delta_{ij}s\right)=\Psi$$
is an important problem in centro-affine differential geometry, known as the $L_{-n}$ Minkowski problem.
See \cite{Chou1,Chen,Chou,Jiang,U} for several important results related to this problem.
  \item  It was pointed out to me by Ben Andrews that it may be possible to extract an inequality for the affine mean curvature from the affine
 Wirtinger inequality which is similar to Michael-Simon Sobolev inequality for the mean curvature, a core tool
 in  minimal surface theory.
\end{enumerate}
\end{remark}
\section{Appendix}
We need the following structure equations
\begin{enumerate}
\item $\displaystyle\xi=-h^{ki}\partial_i\mathcal{K}^{1/(n+1)}X_{.,k}+\mathcal{K}^{1/(n+1)}\nu,$
  \item$\displaystyle\xi_i=-A_i^kX_{.,k},$
  \item $\displaystyle\partial^2_{ij}X=\bar{g}_{ij}+C_{ij}^{k}X_{.,k},$ where $C_{ij}^{k}$
is the cubic form with zero traces.
\cite{LT}.
 \end{enumerate}
 To prove Lemma \ref{lem: ev}, we will use the structure equations without specific mention.
\begin{proof}
We start by proving (1) and (2).
\begin{align*}
\partial_tX_{.,i}=\partial_{i}(F\xi)=\xi\partial_{i}F+F\partial_{i}\xi
=\xi\partial_{i}F-FA_i^kX_{.,k}.
\end{align*}
\begin{align*}
\partial_t \nu&=\langle \partial_t \nu, X_{.,i}\rangle g^{ij}X_{.,j}\\
&=-\langle \nu,\partial_t  X_{.,i}\rangle g^{ij}X_{.,j}\\
&=-\langle\nu, \xi\partial_{i}F-FA_i^kX_{.,k} \rangle g^{ij}X_{.,j}\\
&=-\langle\nu,\xi\rangle g^{ij}\partial_{i}FX_{.,j}\\
&=-\mathcal{K}^{1/(n+1)}g^{ij}\partial_{i}FX_{.,j}.
\end{align*}
\begin{align*}
\partial_tg_{ij}&=\partial_t\langle X_{.,i}, X_{.,j}\rangle\\
&=\langle \partial_{ti}^2X,X_{.,j}\rangle+\langle X_{.,i}, \partial_{tj}^2X \rangle\\
&=\langle \xi\partial_iF-FA_i^kX_{.,k},X_{.,j}\rangle+\langle X_{.,i},\xi\partial_jF-FA_j^kX_{.,k}\rangle\\
&=-FA_i^kg_{jk}-FA_j^kg_{ik}-h^{kl}g_{kj}\partial_l\mathcal{K}^{1/(n+1)}\partial_iF-h^{kl}g_{ki}\partial_l\mathcal{K}^{1/(n+1)}\partial_jF.
\end{align*}
\begin{align*}
\partial_t \det g_{ij}&=(\det g_{ij})g^{ij}\partial_tg_{ij}\\
&=\left(-2HF\det g_{ij}-2\bar{g}^{ij}\partial_i \ln \mathcal{K}^{1/(n+1)}\partial_j F\right)\det g_{ij}\\
&=\left(-2HF\det g_{ij}-2\langle\bar{\nabla}\ln \mathcal{K}^{1/(n+1)},\bar{\nabla}F\rangle_{\bar{g}}\right)\det g_{ij}.
\end{align*}
For the next computation $\bar{\Gamma}_{ij}^{k}$ are the Christoffel symbols related to the affine metric $\bar{g}_{ij}.$
\begin{align*}
\partial_t h_{ij}&=\partial_t\langle \partial_{ij}^2X,\nu\rangle\\
&=\langle\partial^3_{tij}X,\nu\rangle+\langle \partial^2_{ij}X,\partial_t\nu\rangle\\
&=\langle \partial^2_{ij}(F\xi),\nu\rangle+\langle\partial^2_{ij}X,-\mathcal{K}^{1/(n+1)}g^{kl}\partial_kFX_{.,l}\rangle\\
&=\langle \xi\partial^2_{ij}F+F\partial^2_{ij}\xi + \partial_iF\partial_j\xi+\partial_jF\partial_i\xi,\nu \rangle
+\langle \partial^2_{ij}X,-\mathcal{K}^{1/(n+1)}g^{kl}\partial_kFX_{.,l}\rangle\\
&=\langle (\partial^2_{ij}F-\bar{\Gamma}_{ij}^{k}\partial_{k}F)\xi+\bar{\Gamma}_{ij}^{k}\xi\partial_kF,\nu\rangle+
\langle\partial^2_{ij}X-\bar{\Gamma}_{ij}^{k}X_{.,k}+\bar{\Gamma}_{ij}^{k}X_{.,k},-\mathcal{K}^{1/(n+1)}g^{kl}\partial_kFX_{.,l} \rangle\\
&+\langle F\partial_{ij}^2\xi,\nu\rangle\\
&=\mathcal{K}^{1/(n+1)}\bar{\nabla}^{2}_{ij}F+\langle \bar{g}_{ij}\xi+C_{ij}^kX_{.,k},-\mathcal{K}^{1/(n+1)}g^{ml}\partial_mFX_{.,l}\rangle
+\langle F\partial_{ij}^2\xi,\nu\rangle\\
&=\mathcal{K}^{1/(n+1)}\bar{\nabla}^{2}_{ij}F+\bar{g}_{ij}\mathcal{K}^{1/(n+1)}h^{kl}\partial_kF\partial_l\mathcal{K}^{1/(n+1)}-
\mathcal{K}^{1/(n+1)}C_{ij}^n\partial_nF+\langle F\partial_{ij}^2\xi,\nu\rangle\\
&=\mathcal{K}^{1/(n+1)}\bar{\nabla}^{2}_{ij}F+\bar{g}_{ij}\mathcal{K}^{1/(n+1)}h^{kl}\partial_kF\partial_l\mathcal{K}^{1/(n+1)}-
\mathcal{K}^{1/(n+1)}C_{ij}^n\partial_nF+\langle F\partial_i(-A_j^kX_{.,k}),\nu\rangle\\
&=\mathcal{K}^{1/(n+1)}\bar{\nabla}^{2}_{ij}F+\bar{g}_{ij}\mathcal{K}^{1/(n+1)}h^{kl}\partial_kF\partial_l\mathcal{K}^{1/(n+1)}-
\mathcal{K}^{1/(n+1)}C_{ij}^n\partial_nF-\langle FA_j^k\partial^2_{ik}X,\nu\rangle\\
&=\mathcal{K}^{1/(n+1)}\bar{\nabla}^{2}_{ij}F+\bar{g}_{ij}\mathcal{K}^{1/(n+1)}h^{kl}\partial_kF\partial_l\mathcal{K}^{1/(n+1)}-
\mathcal{K}^{1/(n+1)}C_{ij}^n\partial_nF-FA_j^kh_{ik}.
\end{align*}
\begin{align*}
\partial_t\det h_{ij}&=(\det h_{ij})h^{ij}\partial_t h_{ij}\\
&=\left(\bar{\Delta} F+(n-1)\langle \bar{\nabla}\ln\mathcal{K}^{1/(n+1)},\bar{\nabla} F\rangle_{\bar{g}}-FH\right)\det h_{ij}.
\end{align*}
\begin{align*}
\partial_t\mathcal{K}&=\partial_t\left( \frac{\det h_{ij}}{\det g_{ij}}\right)\\
&=2FH\mathcal{K}+2\langle \bar{\nabla}\ln\mathcal{K}^{1/(n+1)},\bar{\nabla} F\rangle_{\bar{g}}\mathcal{K}+\mathcal{K}\bar{\Delta} F+(n-1)\mathcal{K}\langle \bar{\nabla}\ln\mathcal{K}^{1/(n+1)},\bar{\nabla} F\rangle_{\bar{g}}-FH\mathcal{K}\\
&=\mathcal{K}\bar{\Delta} F+\mathcal{K}FH+\langle \bar{\nabla}\mathcal{K},\bar{\nabla} F\rangle_{\bar{g}}.
\end{align*}
Now we proceed to prove (3) and (4).

\begin{align*}
\partial_t s(z)&=\partial_t\langle z, X((-\nu)^{-1}(z),t)\rangle\\
&=\langle z, \nabla X.\partial_t(-\nu)^{-1}+\partial_tX\rangle\\
&=\langle z, F\xi\rangle\\
&=\langle -\nu, F\xi\rangle\\
&=-f.
\end{align*}
\begin{align*}
\partial_t \mathcal{K}&=\partial_t \left(\frac{\det \hat{g}_{ij}}{\det h_{ij}}\right)\\
&=-\left(\frac{\det \hat{g}_{ij}}{\det^2 h_{ij}}\right)\partial_t\det h_{ij}\\
&=-\left(\frac{\det \hat{g}_{ij}}{\det^2 h_{ij}}\right)(\det h_{ij})h^{ij}\partial_t\left(\hat{\nabla}^2_{ij}s+\hat{g}_{ij}s\right)\\
&=\left(\frac{\det \hat{g}_{ij}}{\det h_{ij}}\right)h^{ij}\left(\hat{\nabla}^2_{ij}f+\hat{g}_{ij}f\right)\\
&=\mathcal{K}h^{ij}\left(\hat{\nabla}^2_{ij}f+\hat{g}_{ij}f\right).
\end{align*}
\end{proof}
\textbf{Acknowledgements}:
I would like to thank Alina Stancu, for her comments, suggestions and encouragements.

\bibliographystyle{amsplain}

\end{document}